\documentclass{amsart}
\pdfoutput=1
\newif\iffinal
\finaltrue
\usepackage{a4wide}
\usepackage[utf8]{inputenc}
\usepackage[T1]{fontenc}
\usepackage[ngerman, english]{babel}
\usepackage{algorithm}
\usepackage{algorithmic}
\usepackage{hyperref}
\usepackage{amssymb, amsthm}
\usepackage{mathtools}

\usepackage{textcomp} 
\iffinal
\else
\usepackage[mark]{gitinfo2}

\newcommand{\TODO}[1]%
{\par\fbox{\begin{minipage}{0.9\linewidth}\textbf{TODO:} #1\end{minipage}}\par}
\newcommand{\Anm}[1]%
{\par\fbox{\begin{minipage}{0.9\linewidth}\textbf{Anmerkung:} #1\end{minipage}}\par}
\fi

\DeclareMathOperator{\adj}{adj}

\DeclareMathOperator{\lc}{lc}
\newcommand{\Abar}{\overline{A}}
\DeclarePairedDelimiter{\abs}{\lvert}{\rvert}

\newcommand{\alphabar}{\overline{\alpha}}
\newcommand{\blockcolumn}[2]{\begin{pmatrix}#1\\#2\end{pmatrix}}
\newcommand{\blockrow}[2]{\begin{pmatrix}#1&#2\end{pmatrix}}
\newcommand{\calC}{\mathcal{C}}
\newcommand{\calF}{\mathcal{F}}
\newcommand{\calG}{\mathcal{G}}
\newcommand{\calI}{\mathcal{I}}
\newcommand{\calP}{\mathcal{P}}
\newcommand{\calS}{\mathcal{S}}
\DeclareMathOperator{\diag}{diag}
\renewcommand{\hbar}{\overline{\pmb{h}}}

\DeclareMathOperator{\GL}{GL}
\DeclareMathOperator{\Int}{Int}
\newcommand{\qbar}{\overline{\pmb{q}}}
\newcommand{\Rbar}{\overline{R}}
\newcommand{\Sbar}{\overline{S}}
\newcommand{\Tbar}{\overline{T}}
\newcommand{\ybar}{\overline{\pmb{y}}}
\newcommand{\N}{\mathbb{N}}
\newcommand{\Q}{\mathbb{Q}}
\newcommand{\Z}{\mathbb{Z}}
\newcommand{\nullideal}[3][]{\mathsf{N}^{#1}_{#2}(#3)}
\newcommand{\mccoymoduleshort}[1]{\operatorname{\mathcal{O}}_{#1}}
\newcommand{\mccoymodule}[2]{\mccoymoduleshort{#1}(#2)}
\newcommand{\Mn}[1]{\operatorname{M}_n(#1)}
\newcommand{\successor}{\operatorname{succ}}

\makeatletter
\newtheorem*{rep@theorem}{\rep@title}
\newcommand{\newreptheorem}[2]{%
\newenvironment{rep#1}[1]{%
 \def\rep@title{#2 \ref{##1}}%
 \begin{rep@theorem}}%
 {\end{rep@theorem}}}
\makeatother

\newreptheorem{theorem}{Theorem}

\newtheorem{theorem}{Theorem}
\newtheorem{proposition}{Proposition}[section]
\newtheorem{corollary}[proposition]{Corollary}
\newtheorem{lemma}[proposition]{Lemma}
\theoremstyle{definition}
\newtheorem{definition}[proposition]{Definition}

\newtheorem{example}[proposition]{Example}
\newtheorem*{remark}{Remark}

\author{Clemens Heuberger}
\address{Institut f\"ur Mathematik\\Alpen-Adria-Universit\"at Klagenfurt\\
 Universit\"atsstra\ss e 65--67\\9020 Klagenfurt am W\"orthersee\\Austria}
\email{\href{mailto:clemens.heuberger@aau.at}{clemens.heuberger@aau.at}}
\thanks{C.~Heuberger is supported by the Austrian Science Fund (FWF):
P~24644-N26}

\author{Roswitha Rissner}
\address{Institut für Analysis und Zahlentheorie\\TU Graz\\Kopernikusgasse 
  24\\8010 Graz\\Austria}
\email{\href{mailto:roswitha.rissner@tugraz.at}{roswitha.rissner@tugraz.at}}
\thanks{R.~Rissner is supported by the Austrian Science Fund (FWF): P~27816-N26}

\title{Computing $J$-ideals of a matrix over a principal ideal domain}


\begin{document}

\begin{abstract}
Given a square matrix $B$ over a principal ideal domain $D$ and an ideal $J$
of $D$, the $J$-ideal of $B$ consists of the polynomials $f\in D[X]$ such that
all entries of $f(B)$ are in $J$.
It has been shown that in order to determine all $J$-ideals of $B$ it suffices
to compute a generating set of the $(p^t)$-ideal of $B$ for finitely 
many prime powers $p^t$. Moreover, it is known that a $(p^t)$-ideal is generated by a 
set of polynomials of the form $p^{t-s}\nu_s$ for certain $s\le t$ where each $\nu_s$ is a monic 
polynomial of minimal degree in the $(p^s)$-ideal of $B$.
However, except for the case of diagonal matrices, it was not known how to determine these 
polynomials explicitly.
We present an algorithm which allows us to compute the polynomials $\nu_s$ for general square 
matrices. 
Exploiting one of McCoy's theorems we first compute some set of generators of the $(p^s)$-ideal of 
$B$ which then can be used to determine $\nu_s$.

\end{abstract}

\maketitle

\section{Introduction} 
 If $B \in \Mn{R}$ is a square matrix over a commutative ring $R$ and $J$ is an ideal of $R$, the $J$-ideal of $B$ is defined as 
 \begin{equation*}
   \nullideal{J}{B} = \{f\in R[X] \mid f(B)\in \Mn{J}\}.
 \end{equation*}
These ideals have been introduced in \cite{Rissner2016} and arise naturally in the study of integer-valued polynomials on a matrix $B$ (see  
below in Section~\ref{sec:results:iv-polys}).
 
In case the underlying ring is a principal ideal domain, the structure of $J$-ideals has been studied thoroughly in \cite{Rissner2016}. It has 
been shown that it suffices to compute a finite number of polynomials in order to describe all $J$-ideals of a matrix $B$. As summarized in 
Section~\ref{sec:results:a-ideals}, it suffices to determine a monic polynomial of minimal degree in $\nullideal{(p^t)}{B}$ for a finite number of 
prime powers $p^t$ of $D$. It is further known that these so-called $(p^t)$-minimal polynomials are strongly related to the decomposition of the 
modules 
\begin{equation*}
(D/p^tD)[B + \Mn{p^tD}]=\{f(B + \Mn{p^tD}) \mid f\in (D/p^tD)[X] \}
\end{equation*}
into cyclic submodules with ascending annihilators, see Section~\ref{sec:results:module-decomposition}.
 
However, the characterization of these generating sets given in \cite{Rissner2016} is theoretic. 
Except for diagonal matrices, it was not known until now how to compute $(p^t)$-minimal polynomials. 
This paper is the algorithmic counterpart of~\cite{Rissner2016}. 
Algorithm~\ref{algorithm:everything} determines these polynomials explicitly for general square matrices $B$ with entries in a principal ideal 
domain. The iterative computation consists of two main steps. Given a generating system of the $(p^{t-1})$-ideal of 
$B$, we first determine a set $\calF$ of polynomials such that $\nullideal{(p^t)}{B} = (\calF) + 
p\nullideal{(p^{t-1})}{B}$. We then perform a couple of carefully chosen polynomial long divisions 
to compute a $(p^t)$-minimal polynomial. 

In order to determine the set $\calF$, we use a description of the null ideal of a matrix given by McCoy in \cite[Theorem~54]{McCoy1948} (see 
Lemma~\ref{lemma:mccoy}). This result allows us to translate the question to that of solving a
system of linear equations modulo $p^t$. In order to solve this linear 
system, we present a special lifting technique in Section~\ref{sec:lifting}. The application of this technique to 
the original question is considered in Section~\ref{sec:some-generators}. The topic of Section~\ref{sec:computing-minimal-polynomials} is then the 
computation of a $(p^t)$-minimal polynomial. Next, in Section~\ref{sec:unimportant-primes} we 
explain why the minimal polynomial $\mu_B$ of $B$ is a $(p^t)$-minimal polynomial for all but 
finitely many prime elements $p$. Finally, in Section~\ref{sec:finite-nr-generators} we prove that 
for the remaining prime elements $p$ it suffices to determine a finite number of $(p^i)$-minimal 
polynomials to describe the $(p^t)$-ideals for all $t\ge 0$. 
	
\section{Results}\label{sec:results}

All rings considered in this paper are assumed to be commutative with unity.
For a ring $R$ and positive integers $r$, $s$, the set of $(r\times s)$-matrices 
over $R$ is denoted by $M_{r,s}(R)$ or by $M_r(R)$ if $r=s$.

\subsection{\texorpdfstring{$(a)$}{(a)}-ideals of matrices}\label{sec:results:a-ideals}

Let $D$ be a principal ideal domain with quotient field $K$, $B\in \Mn{D}$ and $(a)$ be an ideal of 
$D$.
The aim is to describe the structure of the \emph{$(a)$-ideal}
\begin{equation*}
  \nullideal{(a)}{B}=\{ f\in D[X]\mid f(B)\equiv 0 \pmod{a}\}
\end{equation*}
of $B$.

If $a=0$, then it is easily seen that
\begin{equation*}
  \nullideal{(0)}{B} = \mu_B D[X]
\end{equation*}
where $\mu_B$ is the minimal polynomial of $B$ over the quotient field $K$ of $D$,
cf.~\cite{Brown:1998:null}.

If $ 0 \ne a = bc$ for coprime elements $b$ and $c$, then $\nullideal{a}{B} = 
c\nullideal{b}{B} + b\nullideal{c}{B}$ according to \cite[Lemma~2.9]{Rissner2016}. Since 
every element in $D$ has a decomposition into primes, it suffices to consider the case $a=p^t$ 
where $p$ is a prime element and $t\in \N$.

For almost all prime elements $p$, we have
\begin{equation*}
  \nullideal{(p^t)}{B}= \mu_B D[X] + p^tD[X]
\end{equation*}
for $t\ge 1$.
More precisely, this is the case for all primes $p$ which do not divide $\det(T)$ 
where $T$ is a matrix in $\Mn{D}\cap \GL_n(K)$ such that $TBT^{-1}$ is in rational canonical form, see Theorem~\ref{theorem:trivial-primes}. 
However, the transformation matrix $T$ is not uniquely determined and the set of prime divisors $\det(T)$ depends on the choice of $T$, see 
Example~\ref{example:candidates-not-unique}.

Thus it is sufficient to determine $\nullideal{(p^t)}{B}$ for finitely many
primes $p$. The following result is a consequence of \cite[Theorem~2.19, 
Corollary~2.23]{Rissner2016}. We give a proof below in Section~\ref{sec:finite-nr-generators}.

\begin{theorem}[{\cite[Theorem~2.19, Corollary~2.23]{Rissner2016}}]\label{theorem:structure} 
Let $p$ be
a prime element of $D$. Then there is a finite set $\calS_p$ of positive integers and monic 
polynomials $\nu_{(p,s)}$ for $s\in\calS_p$  such that for $t\ge 1$,
\begin{equation*}
    \nullideal{(p^t)}{B}=\mu_BD[X] + p^tD[X] + 
      \sum_{\substack{s\in\calS_p \\ s \le  b(t) }} p^{\max\{0,t-s\}}\nu_{(p,s)}D[X]  
  \end{equation*}
  holds where $ b(t) = \inf\{r\in \calS_p \mid r \ge t\}$. The degree of
  $\nu_{(p,s)}$ is strictly increasing in $s\in \calS_p$ and $\nu_{(p,s)}$ is a monic
  polynomial of minimal degree in $\nullideal{(p^s)}{B}$. If $t\le
  \max\calS_p$, then the summand $\mu_BD[X]$ can be omitted.
\end{theorem}

Whereas \cite{Rissner2016} could only show the existence of these $\calS_p$,
and $\nu_{(p,s)}$, $s\in\calS_p$, the present paper
presents an algorithm (Algorithm~\ref{algorithm:everything}) to explicitly compute these quantities. Thus the
structure of $\nullideal{(a)}{B}$ is completely understood.

For simplicity, we omit the indices $p$ and write $\calS$ and $\nu_s$
instead of $\calS_p$ and $\nu_{(p,s)}$, respectively, when the prime $p$ is clear
from the context.

An implementation of Algorithm~\ref{algorithm:everything} has been included~\cite{trac:21992} 
in the free open-source mathematics software system SageMath~\cite{SageMath:2016:7.5} as
method \verb+p_minimal_polynomials+ of a matrix; the $(a)$-ideal of $B$ can be computed by
the method \verb+null_ideal+ of $B$.

\subsection{Integer-valued polynomials}\label{sec:results:iv-polys}

Let $D$ be a principal ideal domain with quotient field $K$ and $B\in\Mn{D}$. Then
\begin{equation*}
  \Int(B, \Mn{D})\coloneqq \{ f\in K[X] \mid f(B)\in\Mn{D}\}
\end{equation*}
is called the ring of integer-valued polynomials on $B$. As before, the minimal
polynomial of $B$ over $K$ is denoted by $\mu_B$.

If a polynomial $f\in K[X]$ is written as $f=g/d$ for some $g\in D[X]$ and $d\in D$,
then $f\in\Int(B, \Mn{D})$ holds if and only if $g\in \nullideal{(d)}{B}$.
Thus Theorem~\ref{theorem:structure} translates into the following corollary
proved in \cite{Rissner2016}.

\begin{corollary}[{\cite[Theorem~4.3]{Rissner2016}}]
  With the above notations, there is a finite set $\calP$ of prime elements such that
  \begin{equation*}
    \Int(B, \Mn{D})= \mu_B K[X] + D[X] + \sum_{p\in\calP}\sum_{s\in\calS_p} 
\frac1{p^{s}}\nu_{(p,s)}D[X]
  \end{equation*}
  where $\calS_p$ and $\nu_{(p,s)}$, $s\in\calS_p$, are the set and polynomials
  from Theorem~\ref{theorem:structure}.
\end{corollary}

As a consequence, Algorithm~\ref{algorithm:everything} completely describes the
structure of $\Int(B, \Mn{D})$.

An implementation has been submitted as method
\verb+integer_valued_polynomials+ for inclusion in SageMath.

\subsection{Module decompositions}\label{sec:results:module-decomposition} 

Again, let $D$ be a principal ideal domain, $B\in\Mn{D}$ and $p^t$ a prime power of $D$. The $(D/p^tD)$-module 
\begin{equation*}
(D/p^tD)[B + \Mn{p^tD}] = \{f(B + \Mn{p^tD}) \mid f\in (D/p^tD)[X] \}
\end{equation*}
is a finitely generated module over a principal ideal ring. According to \cite[Theorem~15.33]{Brown1993} this module 
decomposes into a direct sum of cyclic submodules with uniquely determined annihilators (the invariant factors). As shown in \cite{Rissner2016}, this 
decomposition is strongly related to the generating set described in Theorem~\ref{theorem:structure} which is stated in the next theorem. This 
further implies that the set $\calS_p$ and the degrees of the polynomials $\nu_{(p,s)}$, $s\in 
\calS_p$ are uniquely determined.

\begin{theorem}[{\cite[Theorem~3.5]{Rissner2016}}]
Let $B\in \Mn{D}$ and for a prime $p$ of $D$, let  $\calS_p$ and $\nu_{(p,s)}$, $s\in \calS_p$, the 
set and polynomials from 
Theorem~\ref{theorem:structure}. We order $\{\nu_{(p,s)} \mid s\in\calS_p\} \cup \{\mu_B\}$ by 
ascending degree and define $\successor(\nu_{(p,s)})$ for 
$s\in\calS_p$ to be the successor of $\nu_{(p,s)}$ with respect to this ordering. 
Finally, let $d = \min\{\deg(\nu_{(p,s)})\mid s\in \calS_p\}$ and $d_s = 
\deg(\successor(\nu_{(p,s)})) - \deg(\nu_{(p,s)})$. Then
\begin{equation*}
(D/p^tD)[B +\Mn{p^tD}] = (D/p^{t}D)^d \oplus \bigoplus_{\substack{s\in \calS_p \\ s \le t}} (D/p^{t-s}D)^{d_s}
\end{equation*}
for $t\ge 0$.
\end{theorem}

As a consequence, Algorithm~\ref{algorithm:everything} completely determines the
structure of $(D/p^tD)[B +\Mn{p^tD}]$.

\section{Lifting}\label{sec:lifting}

In this section, we provide the lifting procedure which allows the recursive computation of the $(p^t)$-minimal polynomials in 
Section~\ref{sec:computing-minimal-polynomials}.

Let $D$ be a principal ideal domain, $p$ be a prime element of $D$ and $d\ge c\ge 1$. The projection of some $z\in D$ to the
field $D/pD$ is denoted by $\overline{z}$. We extend this
notation to polynomials in $D[X]$ and vectors in $D[X]^d$ as well as matrices
in $M_{c, d}(D[X])$. The identity matrix is denoted by $I$.

Let $A\in M_{c, d}(D[X])$ be a matrix such that $\Abar$ has rank $c$. For
$t\ge 0$, we consider the set
\begin{equation*}
  \mccoymoduleshort{t}\coloneqq\mccoymodule{t}{A}\coloneqq \{ \pmb{f}\in D[X]^d \mid  A\pmb{f} 
\equiv 0 \pmod{p^t}\}.
\end{equation*}
This is clearly a $D[X]$-module. For $t=0$, we obviously have $\mccoymoduleshort{t}=D[X]^d$.

A recursive method for computing $\mccoymoduleshort{t}$ is given in Algorithm~\ref{algorithm:lifting}.

\begin{algorithm}
  \begin{algorithmic}
    \REQUIRE $t\ge 1$, $G\in M_{d,s}(D[X])$ such that the columns of $\blockrow{p^{t-1}I}{G}$
    are generators of $\mccoymoduleshort{t-1}$
    \ENSURE $F\in M_{d, d-c+s}(D[X])$ such that the columns of
    $\begin{pmatrix}{p^{t}I}&{F}&pG\end{pmatrix}$ are generators of $\mccoymoduleshort{t}$
    \STATE $R\coloneqq \frac{1}{p^{t-1}} AG$
    \STATE Let $S\in M_{c}(D[X])$ and $T\in M_{d+s}(D[X])$ such
    that $\Sbar$ and $\Tbar$ are invertible and $\Sbar\blockrow{\Abar}{\Rbar}
    \Tbar=\diag_{c\times (d+s)}(\alphabar_1,\ldots, \alphabar_c)$ with
    $\alphabar_1\mid \alphabar_2\mid \cdots\mid \alphabar_c$ (Smith normal
    form)
    \STATE Let $F\in M_{d, d-c+s}(D[X])$ consist of the
    last $d-c+s$ columns of $\blockrow{p^{t-1}I}{G}T$
  \end{algorithmic}
  \caption{Recursive computation of $\mccoymoduleshort{t}$}
  \label{algorithm:lifting}
\end{algorithm}

\begin{proposition}\label{prop:lifting}
  Algorithm~\ref{algorithm:lifting} is correct.
\end{proposition}
\begin{proof}
  By hypothesis, an element $\pmb{f}\in D[X]^d$ is in $\mccoymoduleshort{t-1}$ if and only if there 
  exist $\pmb{h}\in D[X]^d$ and $\pmb{q}\in D[X]^s$ such that $\pmb{f}=p^{t-1}\pmb{h}+G\pmb{q}$. 

  Since $\mccoymoduleshort{t}\subseteq \mccoymoduleshort{t-1}$, it follows that  $\pmb{f}\in 
  \mccoymoduleshort{t}$ if and only if $\pmb{f}=p^{t-1}\pmb{h}+G\pmb{q}$ for some $\pmb{h}\in 
  D[X]^d$ and $\pmb{q}\in D[X]^s$ and   $p^{t-1}A\pmb{h}+AG\pmb{q}\equiv0\pmod{p^t}$. As the columns 
  of $G$ are elements of  $\mccoymoduleshort{t-1}$, the matrix $R$ is indeed an element of 
  $M_{c,s}(D[X])$ and  $\pmb{f}\in \mccoymoduleshort{t}$ holds if and only if
  \begin{equation}\label{eq:matrix-congruence}
    A\pmb{h}+R\pmb{q}\equiv 0\pmod{p}.
  \end{equation}
  Projecting into $(D/pD)[X]$, \eqref{eq:matrix-congruence} is equivalent to
  \begin{equation}\label{eq:matrix-equation}
    \blockrow{\Abar}{\Rbar}
    \blockcolumn{\hbar}{\qbar}=0.
  \end{equation}
  As $\Abar$ has full row rank, so do $\blockrow{\Abar}{\Rbar}$ and
  $\diag_{c\times (d+s)}(\alphabar_1,\ldots, \alphabar_c)$.

  Left-multiply \eqref{eq:matrix-equation} by $\Sbar$ to obtain
  \begin{equation*}
    \Sbar\blockrow{\Abar}{\Rbar} \Tbar\,\Tbar^{-1}\blockcolumn{\hbar}{\qbar} =\diag_{c\times 
(d+s)}(\alpha_1,\ldots, \alpha_c)\Tbar^{-1}\blockcolumn{\hbar}{\qbar}=0.
  \end{equation*}
  This is equivalent to
  \begin{equation*}
    \blockcolumn{\hbar}{\qbar}
    =\Tbar\blockcolumn{0}{\ybar}
  \end{equation*}
  for a suitable $\pmb{y}\in D[X]^{d-c+s}$.

  Thus $\pmb{h}=\pmb{h}_0+p \pmb{h}_1$ and $\pmb{q}=\pmb{q}_0+p\pmb{q}_1$ for
  some $\pmb{h}_1\in D[X]^d$, $\pmb{q}_1\in D[X]^s$ and
  \begin{equation*}
    \blockcolumn{\pmb{h}_0}{\pmb{q}_0}=T\blockcolumn{0}{\pmb{y}}.
  \end{equation*}
  Thus we have
  \begin{equation*}
    \pmb{f} =p^{t-1}\pmb{h} + G\pmb{q} =  \blockrow{p^{t-1}I}{G}T\blockcolumn{0}{\pmb{y}} +
    p^t\pmb{h}_1+Gp \pmb{q}_1 =  F\pmb{y} +
    p^t\pmb{h}_1+Gp \pmb{q}_1,
  \end{equation*}
  as claimed.
\end{proof}


\section{Generators of \texorpdfstring{$(p^t)$}{(p-power-t)}-ideals}\label{sec:some-generators}

This section is dedicated to the computation of a generating set of the $(p^t)$-ideal of 
a matrix over a principal ideal domain. Before we go into details, let us recall the 
basic definitions. 

\begin{definition}
Let $R$ be a commutative ring, $J$ an ideal and $B\in \Mn{R}$ be a square matrix. The 
\textit{$J$-ideal of $B$} is defined as
\begin{equation*}
  \nullideal{J}{B} =\nullideal[R]{J}{B} = \{\,f\in R[X] \mid f(B) \in \Mn{J}\,\}.
\end{equation*}
A monic polynomial $\nu \in R[X]$ is called \textit{$J$-minimal polynomial of $B$} if 
\begin{enumerate}
 \item $\nu\in \nullideal{J}{B}$ and
 \item $\deg(g) \geq \deg(\nu)$ for all monic polynomials $g$ with $g\in 
\nullideal{J}{B}$.
\end{enumerate}
We omit the superscript in $\nullideal[R]{J}{B}$ if the underlying ring is clear from the context.
\end{definition}

\begin{remark}
 The $(0)$-ideal of $B$ is just the null ideal of $B$, that is, $\nullideal{(0)}{B} = \{\,f\in R[X]\mid f(B)=0\,\}$. In case that $R$ is a field, 
$\nullideal{(0)}{B}$ is a principal ideal of $R[X]$. The \textit{minimal polynomial} of $B$ is the (in this case) uniquely determined $(0)$-minimal 
polynomial of $B$. Over general commutative rings, a $(0)$-minimal polynomial of a matrix is not
necessarily uniquely determined although its degree is.
\end{remark}

\begin{remark} 
 Note that every square matrix $B\in \Mn{R}$ has a $J$-minimal polynomial for every ideal 
$J$ of $R$. This is due to the Cayley-Hamilton theorem; every matrix over a 
commutative ring is a root of its own characteristic polynomial which is monic, 
cf.~\cite[Theorem~XIV.3.1]{Lang2002}. 
Let $B+\Mn{J}\in \Mn{R/J} $ be the residue class of $B$ modulo $J$ and $\chi\in (R/J)[X]$ denote  
the characteristic polynomial of $B+\Mn{J}$. Then every preimage $f\in R[X]$ of $\chi$ 
satisfies $f(B) \equiv \chi( B+\Mn{J}) \equiv 0 \pmod{\Mn{J}}$ and hence $f\in 
\nullideal{J}{B}$. In particular, there exists a monic preimage of $\chi$ in $R[X]$.
\end{remark}

From now on, let the underlying ring  be the principal ideal domain $D$ and $B\in \Mn{D}$ 
a square matrix. For any ideal $J$ of $D$ there exists $a\in D$ such that $J = (a)$. 
Following the convention in \cite{Rissner2016}, we write $\nullideal{a}{B}$ instead of 
$\nullideal{(a)}{B}$. 

Assume that $a=0$ and let $K$ denote the quotient field of $D$. The null ideal of $B$ considered 
as a matrix over $K$ is generated by its minimal polynomial $\mu_B\in K[X]$. Since the characteristic polynomial $\chi \in D[X]$ of $B$ is in 
$\nullideal[K]{0}{B}$ it follows that $\mu_B$ divides $\chi$. However, $D$ is
integrally closed and therefore every monic factor in $K[X]$ of a monic polynomial 
in $D[X]$ is already an element of $D[X]$ (see \cite[Ch.~5, §1.3, Prop.~11]{Bourbaki_CommAlg}). Hence $\mu_B \in D[X]$ and 
\begin{equation*}
\nullideal[D]{0}{B} = \nullideal[K]{0}{B} \cap D[X]  = \mu_B K[X]\cap D[X] = \mu_B D[X].
\end{equation*}

In order to find a generating set of $\nullideal{p^t}{B}$, we reformulate the problem in a 
form to which the approach of the previous section is applicable. For this purpose we use one
of McCoy's theorems.

\begin{lemma}[{\cite[Theorem~54]{McCoy1948}}]\label{lemma:mccoy}
 Let $R$ be a commutative ring and $C \in \Mn{R}$ a square matrix. Then 
\begin{equation*}
  \nullideal{0}{C}=\{ f\in R[X] \mid \exists Q\in\Mn{R[X]}\colon  \adj(X-C)f(X) = Q(X)\chi_C(X) \}.
\end{equation*}
Here, $\adj(X-C)\in\Mn{R[X]}$ is the adjugate (or classical adjoint) matrix of $X-C$ and $\chi_C\in 
R[X]$ 
denotes the characteristic polynomial of $C$.
\end{lemma}

Since this result is central to our work, we restate its proof here for the reader's convenience.

\begin{proof}
We embed $R[X]$ in $\Mn{R[X]}\simeq\Mn{R}[X]$ via $f \longmapsto f(X)I$ where $I$ is the 
identity matrix and identify $f\in 
R[X]$ with its image. A straight-forward verification shows that $f(C) = 0$ if and 
only if $f(X) \in \Mn{R}[X](X-C)$. 

Being a monic polynomial, $X-C$ is not a zero-divisor in $\Mn{R}[X]$ and therefore
\begin{equation*}
\adj(X-C)f(X) \in \Mn{R}[X]\chi_C(X) 
\end{equation*}
if and only if 
\begin{equation*}
\chi_C(X)f(X)=\adj(X-C)(X-C)f(X) \in \Mn{R}[X]\chi_C(X)(X-C) 
\end{equation*}
which is, in turn, equivalent to 
\begin{equation*}
f(X) \in \Mn{R}[X](X-C) 
\end{equation*}
since $\chi_C(X)$ is also not a zero-divisor in $\Mn{R}[X]$. 
\end{proof}

If $B+\Mn{p^tD}\in \Mn{D/p^tD}$ denotes the residue class of $B$ modulo $p^t$, then 
$\nullideal{p^t}{B}$ is the preimage of $\nullideal{0}{B+\Mn{p^tD}}$ under the 
projection modulo $p^t$. Hence we can write the $(p^t)$-ideal of our matrix $B\in \Mn{D}$ 
in the following way.

\begin{corollary}\label{cor:mccoy-form}
Let $D$ be a principal ideal domain, $p\in D$ a prime element, $B\in \Mn{D}$ be a square 
matrix and $t\ge 0$. Then
\begin{equation*}
  \nullideal{p^t}{B}=\{ f\in D[X] \mid \exists Q\in\Mn{D[X]}\colon  \adj(X-B)f(X) \equiv 
   Q(X)\chi_B(X) \pmod{p^t}\}.
\end{equation*}
\end{corollary}

Note that if $t=0$, then $(p^0) = D$ and $D/D$ is the zero ring 
which has no unity and we cannot apply McCoy's theorem (Lemma~\ref{lemma:mccoy}). However, it is 
easily seen that the corollary still holds which is why we allow the case 
$t= 0$.

For $1\le i, j\le n$, the congruence in Corollary~\ref{cor:mccoy-form} translates to
\begin{equation*}
  \adj(X-B)_{ij} f-\chi_B Q_{ij}\equiv 0\pmod{p^t}.
\end{equation*}

\begin{proposition}\label{proposition:M-N-generators}
Let $D$ be a principal ideal domain, $p\in D$ a prime element, $B\in \Mn{D}$ be a square 
matrix with characteristic polynomial $\chi_B$ and $t\ge 0$. Further, let $\pmb{b}\in 
D[X]^{n^2}$  be the column vector 
containing all entries of $\adj(X-B)$ in some fixed order,
$A=\blockrow{\pmb{b}}{-\chi_BI} \in M_{n^2,n^2+1}(D)$ where  $I$ denotes the $n^2 \times 
n^2$-identity matrix and 
\begin{equation*}
 \mccoymodule{t}{A} = \{ \pmb{f}\in D[X]^{n^2+1} \mid  A\pmb{f} \equiv 0 \pmod{p^t}\}.
\end{equation*}

For $g_{11}$, \ldots, $g_{1s}\in D[X]$ (with $s\in \N_0)$, the following assertions are equivalent:
\begin{enumerate}
 \item There is a matrix $G\in M_{n^2+1,s}(D[X])$ with first row $g_{11}$, \ldots,
   $g_{1s}$ such that the columns of $\blockrow{p^{t}I}{G}$ form a generating set of $\mccoymodule{t}{A}$.
 \item $p^t$, $g_{11}$, $\ldots$, $g_{1s}$ form a generating set of $\nullideal{p^t}{B}$.
\end{enumerate}
\end{proposition}
\begin{proof}
The ideal $\nullideal{p^t}{B}$ is the projection of $\mccoymodule{t}{A}$ on the first component 
according to Corollary~\ref{cor:mccoy-form}. It follows that $(1)$ implies $(2)$. For the reverse 
implication let $g_{11}$, $\ldots$, $g_{1s}\in D[X]$ such that $\nullideal{p^t}{B} = 
(p^t,g_{11},\ldots,g_{1s})$.
Since $g_{1i}\in \nullideal{p^t}{G}$ for $1\le i \le s$, there exist polynomials $g_{ji} 
\in D[X]$ for $2\le j \le n^2+1$ such that $\pmb{g}_i = (g_{ji})_j \in \mccoymodule{t}{A}$ according 
to Corollary~\ref{cor:mccoy-form}. If $\pmb{f} = (f_j)_j\in \mccoymodule{t}{A}$ is an arbitrary 
element, then $f_1\in \nullideal{p^t}{A}$ and there exist $h_0$, $\ldots$, $h_s$ such that 
\begin{equation}\label{eq:rep-f1}
f_1 = p^th_0 + \sum_{i=1}^s h_ig_{1i}.
\end{equation}
If $b_j$ denotes the $j$-th coordinate of $\pmb{b}$ then, by definition of $\mccoymodule{t}{A}$, 
\begin{equation*}
 f_1 b_j \equiv \chi_B f_{j+1} \pmod{p^t}
\end{equation*}
and 
\begin{equation*}
 g_{1i} b_j \equiv \chi_B g_{j+1,i} \pmod{p^t}
\end{equation*}
hold for $1 \le j\le n^2$. Together with \eqref{eq:rep-f1}, we get
\begin{equation*}
\chi_B f_{j+1} \equiv f_1 b_j \equiv \sum_{i=1}^s h_ig_{1i}b_j 
\equiv \sum_{i=1}^s h_i \chi_B g_{j+1,i} \pmod{p^t}.
\end{equation*}
Since $\chi_B$ is monic, its residue class modulo $p^t$ is no zero-divisor and we can 
cancel $\chi_B$ in the equation above to conclude that 
\begin{equation*}
f_{j+1} \equiv \sum_{i=1}^s h_i g_{j+1,i} \pmod{p^t}.
\end{equation*}
Therefore $\pmb{f}$ is a $D[X]$-linear combination of the columns of 
$\blockrow{p^{t}I}{G}$ where $G=(g_{ji})_{\substack{1\le j\le n^2+1\\1\le i\le s}}$.
\end{proof}

It follows from Proposition~\ref{proposition:M-N-generators} that we can use 
Algorithm~\ref{algorithm:lifting} to compute a generating system of $\mccoymodule{t}{A}$ (with 
$c=n^2$ and $d=n^2+1$). Note that the residue class $\overline{A}$ of $A$ has
full rank
since $\chi_B$ is monic. Assume that we are given $G\in M_{n^2+1,s}(D[X])$ whose columns 
together with those of $p^{t-1}I$ generate $\mccoymodule{t-1}{A}$, then 
Algorithm~\ref{algorithm:lifting} computes a matrix $F\in 
M_{n^2+1,s+1}(D[X])$ such that the columns of $\begin{pmatrix}{p^{t}I}&{F}&pG\end{pmatrix}$ 
generate $\mccoymodule{t}{A}$. Then $p^t$, $f_{11}$, $\ldots$, $f_{1,s+1}$, $pg_{11}$,
$\ldots$, $pg_{1s}$ generate 
$\nullideal{p^t}{B}$, according to Proposition~\ref{proposition:M-N-generators}. In particular, the 
elements $f_{11}$, $\ldots$, $f_{1,s+1}$ satisfy a property which motivates the next definition.

\begin{definition}\label{def:pt-generating-property}
Let $D$ be a principal ideal domain, $p\in D$ a prime element, $B\in \Mn{D}$ be a square 
matrix and $t\ge 1$. We say that a finite subset $\mathcal{F}$ of $D[X]$ has the 
\textit{$(p^t)$-generating property w.r.t.~$B$} if 
  \begin{equation*}
   \nullideal{p^{t}}{B} = (\mathcal{F}) + p\nullideal{p^{t-1}}{B}.
  \end{equation*}
\end{definition}

As output of Algorithm~\ref{algorithm:lifting},  $\mathcal{F}=\{f_{11}, \ldots,
f_{1,s+1}\}$ has $s+1$ elements. 
Applying the algorithm recursively leads to a huge set of generators.

Indeed, according to \cite[Proposition~2.13]{Rissner2016}, if $t\ge 1$ and $\nu_t$ is a 
$(p^t)$-minimal polynomial of $B$, then
\begin{equation}\label{eq:recursive-description}
 \nullideal{p^t}{B} = (\nu_t) + p\nullideal{p^{t-1}}{B}.
\end{equation}
This is also a consequence of Algorithm~\ref{algorithm:everything} below.

The next section is dedicated to the question how to compute $(p^t)$-minimal polynomials 
given a set $\mathcal{F}$ with the $(p^t)$-generating property and a $(p^{t-1})$-minimal 
polynomial. For now, we assume that we already know $(p^i)$-minimal polynomials $\nu_i$ 
for $1\le i\le t-1$. Then Equation~\eqref{eq:recursive-description} implies
\begin{equation}\label{eq:nr-of-generators}
 \nullideal{p^{t-1}}{B} = \sum_{i=0}^{t-1} p^{t-1-i}\nu_{i}D[X] 
\end{equation}
and according to Proposition~\ref{proposition:M-N-generators}, there exists a matrix $G\in 
M_{n^2+1,t-1}(D)$ with $g_{1i} = p^{t-1-i}\nu_{i}$ for $1\le i \le t-1$ such that 
$\blockrow{p^{t-1}I}{G}$ generates $\mccoymodule{t-1}{A}$. With this choice, $s = t-1$ 
and hence Algorithm~\ref{algorithm:lifting} produces a set $\mathcal{F}$ with $t$ elements.

Note that \cite[Theorem~2.19]{Rissner2016} states that it suffices to sum over the $(t-1)$-st index 
set in Equation~\eqref{eq:nr-of-generators} which may result in a smaller number of columns of $G$. 
However, even this reduction technique does not yield $\abs{\calF}=1$ except in
trivial cases. Therefore, reduction of $\abs{\calF}$ in every step is essential.

\section{Computing \texorpdfstring{$(p^t)$}{(p-power-t)}-minimal polynomials}\label{sec:computing-minimal-polynomials}

This section considers the question how to compute a $(p^t)$-minimal polynomial of a 
square matrix $B\in \Mn{D}$ over a principal ideal domain $D$ for $t\ge 1$. For this purpose, we assume throughout this 
section that we already determined a $(p^i)$-minimal polynomial $\nu_i$ for $0 \le i \le 
t-1$ and a set $\mathcal{F}$ with the $(p^t)$-generating property. 
We start with a special case, namely the case where the set $\mathcal{F}$  consists of a 
single monic polynomial $f$.

\begin{proposition}\label{proposition:singletons-with-generating-property}
Let $D$ be a principal ideal domain, $p\in D$ a prime element, $B\in \Mn{D}$ be a square 
matrix, $t\ge 1$ and $\nu \in D[X]$ be a monic polynomial.

If $\nullideal{p^t}{B} = (\nu) + p\nullideal{p^{t-1}}{B}$ then $\nu$ is a $(p^t)$-minimal polynomial.
\end{proposition}

\begin{proof}
 Since $\nu\in\nullideal{p^t}{B}$ holds by hypothesis, it suffices to show that $\deg(f) 
\ge \deg(\nu)$ for all monic polynomials $f \in \nullideal{p^t}{B}$. By assumption, 
 \begin{equation*}
  \nullideal{p^t}{B} = (\nu) + p\nullideal{p^{t-1}}{B}
 \end{equation*}
holds, so for a monic polynomial $f\in \nullideal{p^t}{B}$ there exist 
$g\in D[X]$ and $h\in \nullideal{p^{t-1}}{B}$ such that $f = g\nu + ph$. Let
$g_1$, $g_2\in 
D[X]$ be polynomials such that $g = g_1 + pg_2$ and no non-zero coefficient of $g_1$ is 
divisible by $p$. Then $f = g_1\nu + p(g_2\nu + h)$ and since $f$ is monic and $p$ does 
not divide $\lc(g_1) = \lc(g_1 \nu)$ it follows that $\deg(f) = \deg(g_1)+\deg(\nu) \ge 
\deg(\nu)$. 
\end{proof}

In order to apply Proposition~\ref{proposition:singletons-with-generating-property}, we
have to reduce the output set $\mathcal{F}$ of
Algorithm~\ref{algorithm:lifting}
such that it only contains one monic polynomial.

As a first step, observe that
\begin{equation*}
\nullideal{p^t}{B} \cap pD[X] = p\nullideal{p^{t-1}}{B}
\end{equation*}
holds and therefore $\mathcal{F} \setminus p D[X]$ has the $(p^t)$-generating property. 
From now on we can therefore assume that $\mathcal{F} \cap p D[X] = \emptyset$. Since 
$\nullideal{p^t}{B}$ always contains a monic polynomial it follows that 
$\nullideal{p^t}{B} \ne p\nullideal{p^{t-1}}{B}$ and hence $\mathcal{F}$ is never empty.

In order to compute a $(p^t)$-minimal polynomial from the polynomials in $\mathcal{F}$ we 
need the following special case of 
\cite[Corollary~2.14]{Rissner2016}.
\begin{lemma}[{\cite{Rissner2016}}]\label{lemma:degree-lower-bound}
  Let $\nu_{t}$ be a $(p^{t})$-minimal polynomial of $B$ and $f\in
  \nullideal{p^t}{B}$.

  If $f\notin pD[X]$, then $\deg(f) \ge \deg(\nu_{t})$.
\end{lemma}

The idea is to start with $\mathcal{F}_0 = \mathcal{F}$ and show (for $i\ge 1$) that if 
$|\mathcal{F}_{i-1}|>1$ then we can compute a set $\mathcal{F}_i$ of monic
polynomials with the 
$(p^t)$-generating property with $\mathcal{F}_i\cap pD[X] = \emptyset$ such
that $(\min_{f\in\calF_i} \deg f, \abs{\calF_i})$ decreases lexicographically
in each step.

Since the degree of monic polynomials in $\nullideal{p^t}{B}$ is clearly
bounded from below,
we end up with a singleton satisfying the $(p^t)$-generating 
property. Hence, by 
Proposition~\ref{proposition:singletons-with-generating-property} the singleton at the end contains a 
$(p^t)$-minimal polynomial.

It turns out that polynomial division is a useful tool to compute $\mathcal{F}_i$. 
However, we are working in $D[X]$, so we cannot just divide some polynomial by another; we 
want to deal with monic polynomials to guarantee that polynomial division is applicable. 
Algorithms~\ref{algorithm:replace-by-monic-polynomial} and \ref{algorithm:replace-by-monic-polynomials} provide the tools to replace a set with the 
$(p^t)$-generating property by another one which consists only of monic polynomials.

\begin{algorithm}
  \begin{algorithmic}
    \REQUIRE $p\in D$ prime, $t\ge 1$, $\nu_{t-1}$ a $(p^{t-1})$-minimal
    polynomial of $B$ and $f\in \nullideal{p^t}{B} \setminus pD[X]$
    \ENSURE Monic polynomial $h\in \nullideal{p^t}{B}$ with $\deg(h)\le \deg(f)$
    \STATE Write $f=f_1+pf_2$ such that all non-zero coefficients of $f_1$ are
    not divisible by $p$
    \STATE Let $r$ be the remainder of $f_2$ modulo $\nu_{t-1}$ with $\deg
    r<\deg \nu_{t-1}$
    \STATE Choose $u$, $v\in D$ with $u\lc(f_1) + vp=1$
    \STATE $h\coloneqq u(f_1+pr)+vX^{\deg(f_1)-\deg(\nu_{t-1})}p\nu_{t-1}$
  \end{algorithmic}
  \caption{Find monic polynomial}
  \label{algorithm:replace-by-monic-polynomial}
  
\end{algorithm}

\begin{lemma}\label{lemma:find-monic-polynomial}
  Algorithm~\ref{algorithm:replace-by-monic-polynomial} is correct.
\end{lemma}

\begin{proof}

Let $q\in D[X]$ such that $f_2 = q\nu_{t-1} + r$. We set
\begin{equation}
\tilde  h = f_1 + pr = f- q(p\nu_{t-1})  \in \nullideal{p^t}{B}.
\end{equation}
Since $p$ does not divide $\lc(f_1)$,
the leading terms of $f_1$ and $pr$ cannot cancel 
each other out and $\deg(\tilde h) = \max\{\deg(f_1),\deg(r)\}$. On the other hand, 
$\tilde h\in \nullideal{p^t}{B}\setminus pD[X]$ and therefore $\deg(\tilde h) \ge 
\deg(\nu_{t})$ by Lemma~\ref{lemma:degree-lower-bound}. We conclude that
\begin{equation*}
 \deg(\nu_{t-1}) \le \deg(\nu_{t}) \le \deg(\tilde h) = \max\{\deg(f_1),\deg(r)\} \le 
\max\{\deg(f_1),\deg(\nu_{t-1})-1\}  
\end{equation*}
and it follows that $\max\{\deg(f_1),\deg(\nu_{t-1})-1\} =\deg(f_1) = \deg(\tilde h)$ and 
therefore $p \nmid \lc(\tilde h)$.

As $\deg(\tilde h)\ge \deg(\nu_{t-1}) >\deg(r)$, we have $\lc(\tilde
h)=\lc(f_1)$. Thus $h$ is monic.
\end{proof}

\begin{algorithm}
  \begin{algorithmic}
    \REQUIRE  $p\in D$ prime, $t\ge 1$, $f\in \nullideal{p^t}{B} \setminus
    pD[X]$
    \ENSURE Monic polynomials $h_1$, \ldots, $h_s \in \nullideal{p^t}{B}$
    such that 
    \begin{enumerate}
    \item $f \in (h_1,\ldots,h_s) + p\nullideal{p^{t-1}}{B}$ and
    \item $\deg(f) \ge \deg(h_1) > \cdots > \deg(h_s)$.
    \end{enumerate}
    \STATE $i\coloneqq 0$, $f_i=f$
    \WHILE{$f_i\notin pD[X]$}
       \STATE $i\coloneqq i+1$
       \STATE Determine monic $h_i\in\nullideal{p^t}{B}$ with $\deg(h_i) \le
       \deg(f_{i-1})$ (Algorithm~\ref{algorithm:replace-by-monic-polynomial})
       \STATE Let $f_i\in 
D[X]$ be  the  remainder of $f_{i-1}$ modulo $h_i$ with 
$\deg(f_{i}) < \deg(h_i)$.
    \ENDWHILE
    \STATE $s\coloneqq i$
  \end{algorithmic}
  \caption{Replacing by monic polynomials}\label{algorithm:replace-by-monic-polynomials}
\end{algorithm}

\begin{lemma}\label{lemma:replace-by-monic-polynomials}
  Algorithm~\ref{algorithm:replace-by-monic-polynomials} terminates and is correct.
\end{lemma}

\begin{proof}
  The construction implies that $f_{i}\in \nullideal{p^t}{B}$ and $f_{i-1} \in 
(h_i,f_i)$.
  
Further, $\deg(f_{i})<\deg(f_{i-1})$ holds which implies that there exists $s\in \N$ 
such that $f_{s} \in pD[X]$. Hence $f_s \in p\nullideal{p^{t-1}}{B}$ and 
  \begin{equation*}
     f=f_0 \in (h_1,f_1) \subseteq  (h_1,h_2,f_2) \subseteq \cdots \subseteq  
(h_1,h_2,\ldots,h_s,f_s) \subseteq  (h_1,\ldots,h_s) + p\nullideal{p^{t-1}}{B}.
  \end{equation*}
\end{proof}

We can now replace $\mathcal{F}$ by a set with the $(p^t)$-generating property which 
consists only of monic polynomials using Algorithm~\ref{algorithm:replace-by-monic-polynomials}. Note that we need to know a 
$(p^{t-1})$-minimal polynomial to do the necessary computations.  
We are now ready to present Algorithm~\ref{algorithm:computenu} to compute a $(p^t)$-minimal polynomial.
\begin{algorithm}[h!]
  \begin{algorithmic}
    \REQUIRE $t\ge 1$, $\mathcal{F}\subseteq D[X]$ with the $(p^t)$-generating property, 
              $\nu_{t-1}$ a $(p^{t-1})$-minimal polynomial
    \ENSURE $(p^t)$-minimal polynomial $\nu_t$ of $B$.
    \STATE Delete all elements in $\mathcal{F} \cap p D[X]$ from $\mathcal{F}$ and then
           replace non-monic polynomials in $\mathcal{F}$ by monic polynomials using 
           Algorithm~\ref{algorithm:replace-by-monic-polynomials}
    \STATE Let $g\in \mathcal{F}$ be of minimal degree.
    \WHILE{$f\in \mathcal{F}$ with $f\neq g$}
       \STATE $\mathcal{F}\coloneqq\mathcal{F} \setminus \{f\}$
       \STATE Let $q$, $r\in D[X]$ such that $f=qg+r$ and $\deg(r)<\deg(g)$.
       \IF {$r\notin pD[X]$}
         \STATE  Let $h_1$, $\ldots$, $h_s$ be monic polynomials with $r \in 
                 (h_1,\ldots,h_s) + p\nullideal{p^{t-1}}{B}$ 
                 (Algorithm~\ref{algorithm:replace-by-monic-polynomials}). 
         \STATE Set $g\coloneqq h_s$ and $\mathcal{F}\coloneqq \mathcal{F} \cup\{ h_1,\ldots,h_s \}$.
       \ENDIF
    \ENDWHILE
    \STATE $\nu_t\coloneqq g$
  \end{algorithmic}
  \caption{Computation of  a $(p^t)$-minimal polynomial}
  \label{algorithm:computenu}
\end{algorithm}

\begin{proposition}
 Algorithm~\ref{algorithm:computenu} terminates and is correct.
\end{proposition}

\begin{proof}
We will show that in every step, $\calF$ consists of monic polynomials and has the
$(p^t)$-generating property and $(\min_{f\in\calF}\deg(f), \abs{\calF})$
decreases lexicographically in each step.

This implies that the algorithm  computes a singleton with the $(p^t)$-generating property. 
According to Proposition~\ref{proposition:singletons-with-generating-property}, such a singleton  
contains a $(p^t)$-minimal polynomial.

Removing all polynomials in $\calF\cap pD[X]\subseteq p\nullideal{p^{t-1}}{B}$
in the first step does not affect the $(p^t)$-generating property. The same
holds for replacing
non-monic polynomials by Algorithm~\ref{algorithm:replace-by-monic-polynomials}.

Now, let $\mathcal{F}_0$ be the result of this first step in the algorithm  and 
$\mathcal{F}_i$ be the resulting set after $i$ iterations of the while loop. Further, let 
$g_i$ be a polynomial of minimal degree in $\mathcal{F}_i$.
  
Now  assume that $|\mathcal{F}_{i-1}|>1$ and let us have a closer look at the $i$-th 
iteration of the while loop. 
For a polynomial $f\in \mathcal{F}_{i-1}$ with $f\neq g_{i-1}$, the algorithm computes 
the remainder $r$ of $f$ modulo $g_{i-1}$ with $\deg(r) < \deg(g_{i-1})$. Then the 
following holds
  \begin{equation*}
   \nullideal{p^t}{B} = (\mathcal{F}_{i-1} \setminus \{f\}) + (r) + p\nullideal{p^{t-1}}{B}.
  \end{equation*}

 We split into two cases: $r\in pD[X]$ and $r\notin pD[X]$.  If $r\in pD[X]$, then $r\in 
p\nullideal{p^{t-1}}{B}$ and hence $\mathcal{F}_i = \mathcal{F}_{i-1}\setminus \{f\}$ has 
the $(p^t)$-generating property. In this case, $|\mathcal{F}_i|<|\mathcal{F}_{i-1}|$ holds 
and $g_i = g_{i-1}$ is a polynomial of minimal degree in $\mathcal{F}_i$.

 If, however, $r\notin pD[X]$, then the algorithm computes monic polynomials 
$h_1$, $\ldots$, $h_s$ with $\deg(r) \ge \deg(h_1) > \cdots >\deg(h_s)$ and $r \in 
(h_1,\ldots,h_s) + p\nullideal{p^{t-1}}{B}$.
 Hence $\mathcal{F}_i = \{h_1,\ldots,h_s\} \cup \mathcal{F}_{i-1}\setminus \{f\}$ has the 
$(p^t)$-generating property and $g_i = h_s$ is a polynomial of minimal degree in 
$\mathcal{F}_i$. In this case, $\deg(g_i)=\deg(h_s)\le \deg(r)<\deg(g_{i-1})$.
\end{proof}

We conclude this section with Algorithm~\ref{algorithm:everything} that computes the generators of $\nullideal{p^t}{B}$ of a matrix 
$B\in \Mn{D}$ and a prime element $p\in D$ as stated in Theorem~\ref{theorem:structure}, that 
are  $(p^s)$-minimal polynomials $\nu_s$ for indices $s$ of a finite set $\calS$ such that for all $t\ge 1$,
  \begin{equation*}
    \nullideal{p^t}{B}=\mu_BD[X] + p^tD[X] + 
      \sum_{\substack{s\in\calS \\ s \le  b(t) }} p^{\max\{0,t-s\}}\nu_{s}D[X]  
  \end{equation*}
  holds where $ b(t) = \inf\{r\in \calS \mid r \ge t\}$.

\begin{algorithm}[h!]
  \begin{algorithmic}
    \REQUIRE $B\in\Mn{D}$, $p\in D$ prime
    \ENSURE $\calS$, $\nu_{s}$ for $s\in\calS$ (Theorem~\ref{theorem:structure})
    \STATE $\chi_B\coloneqq$ characteristic polynomial of $B$
    \STATE $\mu_B\coloneqq$ minimal polynomial of $B$ over quotient field $K$
    \STATE $\pmb{b}\coloneqq$ the entries of $\adj(X-B)$ in some fixed order
    \STATE $A\coloneqq\blockrow{\pmb{b}}{-\chi_BI}$
    \STATE $t\coloneqq 0$, $\calS\coloneqq \emptyset$,
    $G\coloneqq$ the $((n^2+1)\times 0)$-matrix, $\nu_0\coloneqq 1$
    \WHILE{$\mathit{True}$}
    \STATE $t\coloneqq t+1$
    \STATE {\emph{\small(The columns of $\blockrow{p^{t-1}I}{G}$ generate
    $\mccoymodule{t-1}{A}$)}}
    \STATE Determine $F$ such that
    $\begin{pmatrix}{p^{t}I}&{F}&pG\end{pmatrix}$ are generators of
    $\mccoymodule{t}{A}$ by Algorithm~\ref{algorithm:lifting}
    \STATE $\calF\coloneqq$ first row of $F$
    \STATE $\nu_t\coloneqq$ $(p^t)$-minimal polynomial of $B$ by
    Algorithm~\ref{algorithm:computenu} (using $\nu_{t-1}$)
    \IF{$\deg \nu_t\ge\deg\mu_B$}
    \RETURN {$\calS$, $\nu_{s}$ for $s\in\calS$}
    \ENDIF
    \FOR{$i=1, \ldots, n^2$}
    \STATE $\nu_t b_i=g_i\chi_B+r_i$ (long division)
    \ENDFOR
    \STATE $\pmb{g}\coloneqq(g_i)_{1\le i\le n^2}$
    \IF{ $\deg \nu_t = \deg \nu_{t-1}$}
    \STATE Delete last column of $G$
    \STATE $\calS\coloneqq \calS \setminus \{t-1\}$
    \ENDIF
    \STATE $G\coloneqq
    \begin{pmatrix}
      pG&
      \begin{matrix}
        \nu_t\\\pmb{g}
      \end{matrix}
    \end{pmatrix}$
    \STATE $\calS\coloneqq \calS \cup \{t\}$
    \ENDWHILE
  \end{algorithmic}
  \caption{Computation of $\calS$ and minimal polynomials $\nu_s$ for $s\in\calS$}
  \label{algorithm:everything}
\end{algorithm}

\begin{theorem}\label{theorem:everything-is-correct}
  Algorithm~\ref{algorithm:everything} terminates and is correct.

\end{theorem}
\begin{proof}
By definition, $(\deg(\nu_s))_{s\ge 0}$ is a non-decreasing sequence which is bounded from above by $\deg(\mu_B)$. Hence this sequence eventually 
stabilizes. Moreover, as shown in \cite[Proposition~2.22]{Rissner2016}, the sequence always stabilizes at the value $\deg(\mu_B)$, that is, there 
exists $s_0 \ge 0$ such that $\deg(\nu_s) = \deg(\mu_B)$ for all $s\ge s_0$. This implies that the algorithm terminates. 

For the correctness, observe first that Algorithm~\ref{algorithm:computenu} computes a $(p^t)$-minimal polynomial $\nu_t$ with the $(p^t)$-generating 
property. Next, we explain the choice of $\pmb{g}$. As $\nu_t\in \nullideal{p^t}{B}$, there is some $q_i\in D[X]$ such that $\nu_tb_i\equiv 
\chi_Bq_i\pmod{p^t}$ by Corollary~\ref{cor:mccoy-form}. Thus we have $r_i\equiv 
\chi_B(q_i-g_i)\pmod{p^t}$. As $\chi_B$ is monic, the degree of the right hand side modulo $p^t$ exceeds the degree of $r_i$ unless $q_i\equiv 
g_i\pmod{p^t}$. 

If $\deg(\nu_t) = \deg(\nu_{t-1})$, then $\nu_t$ is a $(p^{t-1})$-minimal polynomial and $\nu_t-\nu_{t-1}$ is a polynomial in $\nullideal{p^{t-1}}{B}$ with degree 
less than $\deg(\nu_{t-1})$. By Lemma~\ref{lemma:degree-lower-bound}, this implies that $\nu_t-\nu_{t-1}\in \nullideal{p^{t-1}}{B} \cap pD[X] = 
p\nullideal{p^{t-2}}{B}$. Hence 
\begin{equation*}
  (\nu_t) + p\nullideal{p^{t-2}}{B} = (\nu_{t-1}) + p\nullideal{p^{t-2}}{B} = \nullideal{p^{t-1}}{B}
\end{equation*}
and thus $\nu_t$ has the $(p^{t-1})$-generating property.
Observe that this is in particular the case if the Algorithm reaches its stopping point, that is, if $\deg(\nu_t) = \deg(\mu_B)$. Then $\mu_B$ is a 
$(p^s)$-minimal polynomial for all $s \ge t$.

The remaining proof consists of repeated application of Proposition~\ref{proposition:M-N-generators} in both directions.
\end{proof}

\begin{remark}
Gröbner bases provide an alternative to Algorithm~\ref{algorithm:computenu}. Adams and Loustaunau describe in \cite[Ch.~4.5]{AdamsLoustaunau1994} how 
to extend the theory of Gröbner bases to polynomial rings over principal ideal domains. If $\calG$ is a Gröbner basis of $\nullideal{p^t}{B}$, we set
$\widehat \calG = \{ g\in \calG \mid p\nmid\lc(g) \}$. One can show that $\widehat\calG \neq \emptyset$ and if $g$ is a polynomial of minimal degree 
in $\widehat \calG$, then $ug + vp^tX^{\deg(g)}$ is a $(p^t)$-minimal polynomial where $u,v\in D$ such that $u\lc(g) + vp^t = 1$.
\end{remark}

\begin{example}
We demonstrate Algorithm~\ref{algorithm:everything} and compute $(2^t)$-minimal polynomials for the matrix 
\begin{equation*}
B =
\left(\begin{array}{rrr}
1 & 0 & 1 \\
1 & -2 & -1 \\
10 & 0 & 0
\end{array}\right)
\end{equation*}
for $t\ge 1$. The minimal and characteristic polynomial of $B$ is $\mu_B = \chi_B = X^{3} + X^{2} - 12X - 20$. We skip the computation of $\nu_1$ 
and claim that $\nu_1 = X^2 + X$ is a $(2)$-minimal polynomial (one can check that its residue class is the minimal 
polynomial of $\overline{B}$ over the field $\Z/2\Z$). 

Hence $\nullideal{1}{B} = \nu_1\Z[X] + 2\Z[X]$ and there exists a vector $\pmb{g} \in \Z[X]^{9}$ such that the columns of 
$G_1 = \blockrow{2I}{
\begin{matrix}
 \nu_1\\\pmb{g}
\end{matrix}
}$
generate $\mccoymodule{1}{A}$ where $A = \blockrow{\pmb{b}}{-\chi_BI}$ with 
\begin{equation*}
 \pmb{b} = 
  (X^{2} + 2X , 0 , X + 2 , X - 10 , X^{2} - X - 10 , -X + 2 , 10X + 20 , 0 , X^{2} + X - 2)^t
\end{equation*}
(cf.~Proposition~\ref{proposition:M-N-generators}). Algorithm~\ref{algorithm:everything} performs polynomial long divisions to determine 
\begin{equation*}
 \pmb{g} = 
 (X + 2 , 0 , 1 , 1 , X - 1 , -1 , 10 , 0 , X + 1)^t.
\end{equation*}

Next, Algorithm~\ref{algorithm:everything} calls
Algorithm~\ref{algorithm:lifting} to compute a matrix $F_2\in M_{10,2}(\Z[X])$
such that the columns of
$\begin{pmatrix}{4I}&{F_2}&2G_1\end{pmatrix}$ generate $\mccoymodule{2}{A}$. Without giving details here, we claim that 
\begin{equation*}
 F_2= 
\left(\begin{array}{rrrrrrrrrr}
2X^{2} + 2X & 2X & 0 & 2 & 2 & 2X + 2 & 2 & 0 & 0 & 2X + 2 \\
X^{2} + 3X + 2 & X + 4 & 0 & 1 & 1 & X + 1 & -1 & 10 & 0 & X + 3
\end{array}\right)^t
\end{equation*}
is such a matrix. Hence $\{2X^{2} + 2X, X^{2} + 3X + 2\}$ is a set with the $(4)$-generating property. We can apply 
Algorithm~\ref{algorithm:computenu} which removes the first polynomial as it is an element of $2\Z[X]\cap 
\nullideal{2}{B} = 2\nullideal{1}{B}$. Hence $\{X^{2} + 3X + 2\}$ has the $(4)$-generating property and by 
Proposition~\ref{proposition:singletons-with-generating-property}, $\nu_2 = X^{2} + 3X + 2$ is a $(4)$-minimal polynomial. 
 
If $\pmb{f}$ denotes the second column of $F_2$, then the columns of $G_2 = \blockrow{4I}{\pmb{f}}$ generate 
$\mccoymodule{2}{A}$. In the next step, we apply again Algorithm~\ref{algorithm:lifting} to compute 
\begin{equation*}
F_3=
\left(\begin{array}{rrrrr}
X^{3} + 7X^{2} + 6X & X^{2} + 8X + 24 & 0 & X + 8 & \ldots \\
X^{3} + 3X^{2} + 2X & X^{2} + 4X + 16 & 0 & X + 4 & \ldots
\end{array}\right)^t
\end{equation*}
such that the columns of $\begin{pmatrix}{8I}&{F_3}&2G_2\end{pmatrix}$ generate 
$\mccoymodule{3}{A}$. It follows that $\{X^{3} + 7X^{2} + 6X, X^{3} 
+ 3X^{2} + 2X\}$ has the $(8)$-generating property. Since 
\begin{equation*}
 X^{3} + 7X^{2} + 6X \in  (X^{3} + 3X^{2} + 2X) + 2\nullideal{4}{B},
\end{equation*}
it follows that 
$\nu_3 = X^{3} + 3X^{2} + 2X$ is an $(8)$-minimal polynomial. However, since the degree of $\nu_3$ is equal to $\deg(\mu_B)$, it follows that $\mu_B$ 
is a $(2^t)$-minimal polynomial for $t\ge 3$. Note that $\calS_2 = \{2\}$.
\end{example}

\subsection{Run-time and memory usage in practice}

Table~\ref{table:run-time-mem} displays average run-time and memory usage for 
Algorithm~\ref{algorithm:everything} for a dense random integer matrix $B$ of size $n$ and a prime 
number $p$. Note that only instances with non-trivial $(p^t)$-minimal polynomials were taken 
into account, see Section~\ref{sec:unimportant-primes} below. 
To find such instances, Theorem~\ref{theorem:trivial-primes} below provides a strategy to test only 
a finite number of primes $p$ for a given matrix $B$. Table~\ref{table:run-time-mem} also 
contains the total number of pairs $(B,p)$ to which we applied Algorithm~\ref{algorithm:everything} 
and the number of pairs $(B,p)$ among them with non-trivial $(p^t)$-minimal polynomials.

All computations were done in the free open-source mathematics software system 
SageMath (Version 7.6.beta6) on a machine with an Intel(R) Core(TM) i5-4690S CPU @ 3.20GHz 
processor. 

However, the current implementation of the Smith normal form in SageMath is designed to deal with 
matrices in general principal ideal domains and does not exploit the Euclidean structure 
of univariate polynomial rings over fields. We experienced memory issues using this 
implementation in Algorithm~\ref{algorithm:lifting}. For this reason we implemented the algorithm 
presented in \cite{KaBa:1979:snf} which is also applicable to matrices with entries in a univariate 
polynomial ring over a field.

In addition, it is worth mentioning that large prime numbers can cause a significant increase in 
run-time and memory usage. For example, for $p = 366388788500439413183777$  
Algorithm~\ref{algorithm:everything} takes $5006.08$ seconds and $207.5$ MB given the input matrix 
\begin{equation*}
B =
\begin{pmatrix}
 -1 & -1 &  6 &  3 &  3 & -1 & 11 & -2 & -1 & -2 \\
 -2 & -1 &  9 & -1 & -2 &  1 &  1 &  3 & -1 & -2 \\
 -1 & -6 & -6 & -1 & -4 &  5 &  1 &  1 & -4 &  1 \\
  1 &  1 &  1 & -2 &  1 &  2 & -1 &  1 &  1 & 13 \\
 -1 &  1 & -1 &  3 & -2 & -4 & -1 & -1 &  4 & -4 \\
 -4 &  1 &  1 &  2 & -1 &  2 &  5 & -2 & -1 &  1 \\
 14 &  1 &  1 & -1 &  1 &  2 &  1 &  3 &  1 & -1 \\
 -3 &  1 & -1 &  1 & -3 &  4 & -2 &  2 &  6 & 11 \\
 -2 &  1 & -1 &  1 &  1 &  1 &  6 &-23 & -1 &  1 \\
 -1 & -1 & -1 &  1 &  3 & -1 & -3 &  1 &  1 & -2
\end{pmatrix}
\in M_{10}(\Z).
\end{equation*}
In this particular example however, it turns out that this value of $p$ only
occurs as a root of the determinant of the transformation matrix (see 
Theorem~\ref{theorem:trivial-primes}). In fact, $B$ has no non-trivial 
$(p^t)$-minimal polynomial for any $p$.
As we decided to only include matrices with non-trivial $(p^t)$-minimal
polynomials, this matrix (along with many other examples) does not contribute to
the timings.

\begin{table}[h]
 \begin{tabular}{|*{12}{c|}}
\hline
$n$   & 2 & 3 & 4 & 5 & 6 & 7 & 8 & 9 & 10 & 11\\
\hline
sec   & 0.03 & 0.12 & 0.57 & 2.53 & 8.95 & 28.81 & 80.51 & 193.22 & 501.31 & 983.66\\
\hline
MB & 0.014 & 0.041 & 0.073 & 0.151 & 0.276 & 0.611 & 0.331 & 0.355  & 0.25 & 2.603 \\
\hline
non-trivial   & 73 & 175 & 242 & 329 & 184 & 112 & 111 & 19 & 5 & 13 \\
\hline
total & 587 & 1624 & 2651 & 3571 & 2009 & 1389 & 1489 & 257 & 114 & 106\\
\hline
 \end{tabular}
 \label{table:run-time-mem}
\caption{Run-time in seconds and memory usage in megabytes of 
Algorithm~\eqref{algorithm:everything} for  primes $p$ and integer matrix instances of size $n$
with non-trivial $(p^t)$-minimal polynomials.}
\end{table}

\section{Primes with trivial \texorpdfstring{$(p^t)$}{(p-power-t)}-minimal 
polynomials}\label{sec:unimportant-primes}

In this section, we show that for all but finitely many prime elements $p$ and all $t\ge 1$, $\mu_B$ is a $(p^t)$-minimal polynomial. This further implies that the $(p^t)$-ideal 
$\nullideal{p^t}{B}$ of $B$ is generated by $\mu_B$ and the constant $p^t$.

This assertion has been shown before as auxiliary result in the proof of \cite[Theorem~4.3]{Rissner2016}. In order to make it 
more accessible, we restate it here as proper theorem together with a proof.

\begin{theorem}[{\cite{Rissner2016}}]\label{theorem:trivial-primes}
Let $D$ be a principal ideal domain and $B\in \Mn{D}$ a square matrix with minimal polynomial $\mu_B$. Then for all but finitely many prime elements $p\in D$ and all $t\ge 1$
\begin{equation*}
 \nullideal{p^t}{B} = \mu_B D[X] + p^t D[X].
\end{equation*}
\end{theorem}

\begin{proof}
It suffices to show that $\mu_B$ is a $(p)$-minimal polynomial for all but finitely many prime elements $p\in D$. If this is the case (for a fixed $p$), then Algorithm~\ref{algorithm:everything} stops in the first iteration of the while 
loop returning $\calS = \emptyset$. Hence, $\nullideal{p^t}{B} = \mu_B D[X] + p^t D[X]$ for all $t\ge 0$ by 
Theorem~\ref{theorem:everything-is-correct}. 

Considered as a matrix over the quotient field $K$ of $D$, $B$ is similar to its rational canonical form $C$, that is, 
there exists a matrix $T\in \GL_n(K)$ such that
\begin{align}\label{eq:similarity-over-K}
 TBT^{-1} = C = \calC_{\mu_B} \oplus \cdots \oplus \calC_{\mu_1} 
\end{align}
where $\mu_1\mid\cdots\mid\mu_r = \mu_B$ are the invariant factors of $B$ (in $K[X]$) and $\calC_{\mu_i}$ denotes the companion matrix of the polynomial 
$\mu_i$ for $1\le i \le r$ (cf.~\cite[Theorem~XIV.2.1]{Lang2002}). 
Since $D$ is integrally closed, every monic factor in $K[X]$ of a monic polynomial in $D[X]$ is already in $D[X]$, 
cf.~\cite[Ch.~5,~§1.3,~Prop.~11]{Bourbaki_CommAlg}. Therefore $\mu_i\in D[X]$ for $1\leq i\leq r$ because $\mu_i \mid \mu_B$ and $\mu_B$ divides the 
characteristic polynomial $\chi_B\in D[X]$.

Hence the rational canonical form $C$ of $B$ is a matrix with entries in $D$. Moreover, we can choose $T \in \Mn{D}$. However, in general, the 
similarity relation of $B$ and $C$ does not hold over the domain $D$, that is, we cannot assume $T\in \GL_n(D)$.  

Let $p$ be a prime element that does not divide $\det(T)$. Then $\det(T)$ is invertible in the localization $D_{(p)}$ of $D$ at $p$ and  $T^{-1} = 
\det(T)^{-1} \adj(T) \in \GL_n(D_{(p)})$. This allows to reduce Equation~\eqref{eq:similarity-over-K} modulo $p$
\begin{align*}
 \overline{T}\,\overline{B}\,\overline{T}^{-1} = \overline{TBT^{-1}} = \overline{C} = \calC_{\overline{\mu_B}} \oplus \cdots \oplus 
\calC_{\overline{\mu_1}} 
\end{align*}
(where we identify the residue fields of $D$ and $D_{(p)}$ modulo $p$). Hence $\overline{C}$ is the rational canonical form of $\overline{B}$
which implies that $\overline{\mu_B}$ is the minimal polynomial of $\overline{B}$. Equivalently, $\mu_B$ is a $(p)$-minimal polynomial of $B$. 
The assertion follows since $\det(T)$ has only finitely many prime divisors.
\end{proof}

The choice of the transformation matrix in the proof of Theorem~\ref{theorem:trivial-primes} is not unique. Moreover, the prime divisors of different 
transformation matrices may not coincide as the following example demonstrates.

\begin{example}\label{example:candidates-not-unique}
Let
$B = \begin{pmatrix}
       4 & 5 \\
       3 & 5
     \end{pmatrix} \in M_2(\Z)$. 
The rational canonical form of $B$ is 
$C =
\begin{pmatrix}
    0 & -5 \\
    1 & 9
\end{pmatrix}$.
The matrices 
$T = \begin{pmatrix}
       3 & -4 \\
       0 & 1
     \end{pmatrix}$
and      
$S = \begin{pmatrix}
        1  & 2 \\
        -2 & -3
     \end{pmatrix}$
both satisfy (over $\Q$)
\begin{equation*}
 TBT^{-1} = C = SBS^{-1}.
\end{equation*}
Since $\det(S) = 1$, $B$ is similar to $C$ over $\Z$. This implies that $\mu_B$ is a $(p)$-minimal polynomial for all primes $p$ of $\Z$. However, 
$\det(T) = 3$.
\end{example}

\section{Finite description of \texorpdfstring{$(p^t)$}{(p-power-t)}-ideals for all 
\texorpdfstring{$t$}{(t)}}
\label{sec:finite-nr-generators}

Finally, we give a proof of Theorem~\ref{theorem:structure} which has been stated above 
in Section~\ref{sec:results}. For the reader's convenience we restate it at this point.

\begin{reptheorem}{theorem:structure}[{\cite[Theorem~2.19, Corollary~2.23]{Rissner2016}}]
Let $p$ be a prime element of $D$. Then there is a finite set $\calS_p$ of positive integers and 
monic polynomials $\nu_{(p,s)}$ for $s\in\calS_p$  such that for $t\ge 1$,
\begin{equation*}
    \nullideal{(p^t)}{B}=\mu_BD[X] + p^tD[X] + 
      \sum_{\substack{s\in\calS_p \\ s \le  b(t) }} p^{\max\{0,t-s\}}\nu_{(p,s)}D[X]  
  \end{equation*}
  holds where $ b(t) = \inf\{r\in \calS_p \mid r \ge t\}$. The degree of
  $\nu_{(p,s)}$ is strictly increasing in $s\in \calS_p$ and $\nu_{(p,s)}$ is a monic
  polynomial of minimal degree in $\nullideal{(p^s)}{B}$. If $t\le
  \max\calS_p$, then the summand $\mu_BD[X]$ can be omitted.
\end{reptheorem}

\begin{proof}
It has been shown in \cite[Theorem~2.19]{Rissner2016} that 
\begin{equation}\label{eq:explain-struc-1}
 \nullideal{(p^t)}{B} = \sum_{i\in \calI_t} p^{t-i}\nu_{(p,i)}D[X]
\end{equation}
where $\calI_t$ denotes the $t$-th index set of $B$ with respect to $p$ and $\nu_{(p,i)}$ are monic 
polynomials of minimal degree in $\nullideal{(p^i)}{B}$ whose degree is strictly increasing in  
$i\in \calI_t$. Moreover, it follows from 
\cite[Corollary~2.23]{Rissner2016} that for every $p$ there exists an integer 
$m$ such that
\begin{equation}\label{eq:explain-struc-2}
 \nullideal{(p^t)}{B} = \mu_B D[X] + p^{t-m}\nullideal{(p^m)}{B} 
\end{equation}
holds for all $t\ge m$. We set $\calS_p = \calI_m\setminus \{0,m\}$. Note that $\nu_{(p,0)}= 1$ is 
a monic polynomial of minimal degree in $\nullideal{(p^0)}{B} = D[X]$. For $t\ge m$, the assertion 
now follows from Equations~\eqref{eq:explain-struc-1} and~\eqref{eq:explain-struc-2}.

If $t<m$, it follows from \cite[Definition~2.16, Remark~2.18]{Rissner2016} that  $\calI_t\setminus 
\{0,t\} = \calS_p \cap \{1,\ldots,t-1\}$ and $\nu_{(p,b(t))}$ is also a feasible choice for 
$\nu_{(p,t)}$. Therefore, the assertion follows from Equation~\eqref{eq:explain-struc-1}.
\end{proof}

\bibliographystyle{abbrv}
\bibliography{biblio}  

\begin{thebibliography}{10}

\bibitem{AdamsLoustaunau1994}
W.~W. Adams and P.~Loustaunau.
\newblock {\em An introduction to {G}r\"obner bases}, volume~3 of {\em Graduate
  Studies in Mathematics}.
\newblock American Mathematical Society, Providence, RI, 1994.

\bibitem{Bourbaki_CommAlg}
N.~Bourbaki.
\newblock {\em Commutative Algebra, Chapters 1--7}.
\newblock Springer, Berlin, 1989.

\bibitem{Brown1993}
W.~C. Brown.
\newblock {\em Matrices over Commutative Rings}.
\newblock Monographs and Textbooks in Pure and Applied Mathematics. Marcel
  Dekker, Inc., New York, 1993.

\bibitem{Brown:1998:null}
W.~C. Brown.
\newblock Null ideals and spanning ranks of matrices.
\newblock {\em Comm. Algebra}, 26(8):2401--2417, 1998.

\bibitem{trac:21992}
C.~Heuberger and R.~Rissner.
\newblock Compute {$J$}-ideal of a matrix.
\newblock \url{http://trac.sagemath.org/ticket/21992}, 2016.

\bibitem{KaBa:1979:snf}
R.~Kannan and A.~Bachem.
\newblock Polynomial algorithms for computing the {S}mith and {H}ermite normal
  forms of an integer matrix.
\newblock {\em SIAM J. Comput.}, 8(4):499--507, 1979.

\bibitem{Lang2002}
S.~Lang.
\newblock {\em Algebra}, volume 211 of {\em Graduate Texts in Mathematics}.
\newblock Springer-Verlag New York, 2002.

\bibitem{McCoy1948}
N.~H. McCoy.
\newblock {\em Rings and ideals}.
\newblock Carus Monograph Series, no. 8. The Open Court Publishing Company,
  LaSalle, Ill., 1948.

\bibitem{Rissner2016}
R.~Rissner.
\newblock Null ideals of matrices over residue class rings of principal ideal
  domains.
\newblock {\em Linear Algebra Appl.}, 494:44--69, 2016.

\bibitem{SageMath:2016:7.5}
{The SageMath Developers}.
\newblock {\em {SageMath} {M}athematics {S}oftware ({V}ersion 7.5)}, 2017.
\newblock \url{http://www.sagemath.org}.

\end{thebibliography}
  
\end{document}